\documentclass[12pt]{amsart}

\usepackage{amsthm}

\usepackage{amssymb}
\usepackage{verbatim}
\usepackage[curve,matrix,arrow]{xy}
\usepackage{amsmath,amsfonts}
\usepackage{graphicx}
\usepackage{epsf}

\newcommand{\thmref}[1]{Theorem~\ref{#1}}
\newcommand{\propref}[1]{Proposition~\ref{#1}}

\newcommand{\remref}[1]{Remark~\ref{#1}}

  {\end{list}}

\newtheorem{theorem}{Theorem}[section]

\newtheorem{proposition}[theorem]{Proposition}

\newtheorem{remark}[theorem]{Remark}

\def\mod{\text{mod}}

\begin{document}
\title[Binomial Coefficients and Primality]{The Sufficient Set of Congruences of Binomial Coefficients for Primality}

\author{Zubeyir Cinkir}
\address{Zubeyir Cinkir\\
Department of Industrial Engineering\\
Abdullah Gul University\\
38100, Kayseri, TURKEY\\}
\email{zubeyir.cinkir@agu.edu.tr}

\keywords{Prime number, 
Primality testing, 
Binomial Coefficients, Congruence}

\begin{abstract}
For any given integer $p$, we show that primality of $p$ is equivalent to $\lceil \frac{1}{2}\log_{2}p \rceil$ congruences of binomial coefficients modulo $p$. We also give extensions of this result to several other congruences that are similar in nature.
\end{abstract}

\maketitle
MSCI  05A10,  11A41, 11B65,  11A07

\section{Introduction}\label{sec introduction}

Binomial coefficients as the building blocks of Pascal's triangle has the following fascinating connection to prime numbers:
\begin{theorem}\cite[Pg. 91]{D}\label{thm Leibniz}
A positive integer $p$ is a prime number if and only if $p$ divides $\binom{p}{i}$ for all $0<i<p$.
\end{theorem}
This is a well-known basic fact. In fact, it can be seen as the underlying reason for various primality tests.
For example, for any positive integer $a$, the Binomial Theorem and \thmref{thm Leibniz} can be used together
to show
$$
(a+1)^p = \sum_{i=0}^{p} \binom{p}{i}a^i \equiv a^p+1 \quad  \mod \, \, \, p, \qquad \text{if $p$ is a prime number}.
$$
For a prime number $p$,  taking $a=1$ gives $2^p \equiv 2 \quad  \mod \, \, \, p$. As a follow up of this congruence, 
taking $a=2$ gives $3^p \equiv 2^p +1 \equiv 3 \quad  \mod \, \, \, p$. In the same way, any positive integer $a$ satisfies the following congruence:
$$a^p \equiv a \quad  \mod \, \, \, p, \qquad \text{if $p$ is a prime number}.$$
This is known as Fermat's Little Theorem \cite{PR}. This leads to Fermat's Primality Test (\cite[Page 967]{CLR} and \cite[Chapter 5]{B}).

Since $\binom{p}{i}=\binom{p-1}{i}+\binom{p-1}{i-1}$ for each $0 \leq i \leq p$ and that $\binom{p-1}{0}=1$,
we can derive the following result from \thmref{thm Leibniz}:
\begin{theorem}\label{thm Leibniz2}
A positive integer $p$ is prime number if and only if $\binom{p-1}{i} \equiv (-1)^i \quad \mod \, \, \, p$ for all $0 \leq i \leq p-1$.
\end{theorem}
In fact, the congruence in \thmref{thm Leibniz2} was proved by Lucas \cite{L} in 1879 for prime numbers.

It is known that \thmref{thm Leibniz2} also holds if the range for $i$ may be shortened to $0 \leq i \leq \sqrt{p}$ \cite[pg 1]{CG}.
However, if one wants to use those congruences of binomial coefficients as a deterministic primality test for any given integer $n$, 
it won't be practical. Because, $\binom{n}{i}$ must be computed modulo $n$, and this must be done for $\sqrt{n}$ values of $i$.

Our main contribution in this paper is to show that \thmref{thm Leibniz2} also holds if we just consider $\lceil \frac{1}{2}\log_{2} n \rceil$ values of $i$. Namely, we showed the following result in \thmref{thm p-1 main}, \thmref{thm p-2 maina} and \thmref{thm p-3 main1}:
\begin{theorem}\label{thm main summary}Let $p>4$ be an integer, and let $M$ be the smallest integer with $\sqrt{p} \leq 2^M \leq p-1$. Then,
we have
\begin{enumerate}
\item $p$ is a prime number if and only if  $\binom{p-1}{2^m}   \equiv 1 \quad \mod \, \, \, p$ for each $m \in \{ 1, \, 2, \dots, M \}$.
\item $p$ is a prime number if and only if  $\binom{p-2}{2^m-1}  \equiv 2^m \, \, \, \mod \, \, \, p$ for each $m \in \{ 1, \, 2,  \dots, M \}$.
\item $p$ is a prime number if and only if  $\binom{p-3}{2^m-1}  \equiv -2^{m-1}(2^m+1) \quad \mod \, \, \, p$ for each $m \in \{ 1, \, 2, \dots, M \}$.
\end{enumerate}
\end{theorem}
Moreover, we proved analogous results in \thmref{thm p-1 main2}, \thmref{thm p-1 main3}, \thmref{thm p-2 mainb},  \thmref{thm p-2 mainc}, 
\thmref{thm p-3 main3}, and \thmref{thm p-3 main2}.

\section{Congruences of Binomial Coefficients} \label{sec cong}
First, we recall the following elementary fact:
\begin{remark}\label{rem1}
Let $n$ and $t$ be any two positive integers. Suppose $k$ is an integer satisfying $0 < k < n$ and $\gcd{(n,k)}=1$, then $k$ is invertible in $Z_n^*$. Thus, we have the following congruence:
\begin{equation*}\label{eqn cong1}
\begin{split}
\frac{n t-k}{k}  \equiv -1 \quad \mod \quad n.
\end{split}
\end{equation*}
\end{remark}
Next, we recall the following basic fact:
\begin{proposition}\label{prop basic}
Let $n>1$ be an integer. If $n$ has no prime divisor $p$ with $p \leq \sqrt{n}$, then $n$ is a prime.
\end{proposition}

We state one of the main results of this paper as follows:
\begin{theorem}\label{thm p-1 main} Let $p>2$ be an integer, and let $M$ be the smallest integer with $\sqrt{p} \leq 2^M \leq p-1$. We have
$$\displaystyle{\text{$p$ is a prime number if and only if } \binom{p-1}{2^m}   \equiv 1 \quad \mod \, \, \, p \text{ for each $m \in \{ 1, \, 2, \, \dots, M \}$}.}$$
\end{theorem}
\begin{proof}
We'll prove both directions.

\textbf{Case I ($\Rightarrow$):}
If $p$ is a prime number, then \thmref{thm Leibniz2} gives $\binom{p-1}{2^m}   \equiv (-1)^{2^m}  \equiv 1 \quad \mod \, \, \, p$ 
for each $m \in \{ 1, \, 2, \, \dots, M \}$.

\textbf{Case II ($\Leftarrow$):}
We'll prove the contraposition. Namely, we prove that if $p$ is not a prime number, then $\binom{p-1}{2^m}   \not \equiv 1 \quad \mod \, \, \, p \text{ for some $m \in \{ 1, \, 2, \, \dots, M \}$}.$

If $p$ is not a prime number, then it has a prime divisor $q$ such that $q \leq \sqrt{p}$ by \propref{prop basic}. Then, $p=q^k u$ for some integers $k \geq 1$ and
$u \geq 1$ with $\gcd(q,u)=1$. Since $q \leq \sqrt{p}$, there is a positive integer $m \leq M$ such that $2^{m-1} < q \leq 2^m < 2q$. 
Then,
\begin{equation*}\label{eqn cong3a}
\begin{split}
\binom{p-1}{2^m} &= \frac{p-1}{1} \cdot \frac{p-2}{2} \cdots \frac{p-q+1}{q-1} \cdot \frac{p-q}{q} \cdot \frac{p-q-1}{q+1}\cdots
\frac{p-2^{m}}{2^{m}}\\
&= \frac{q^k u-1}{1} \cdot \frac{q^k u-2}{2} \cdots \frac{q^k u-q+1}{q-1} \cdot \frac{q^k u-q}{q} \cdot \frac{q^k u-q-1}{q+1}\cdots
\frac{q^k u-2^{m}}{2^{m}}\\
& \equiv (-1).(-1) \dots (-1) . (q^{k-1} u-1). (-1) \dots (-1) \quad \mod \, \, \, q^k \quad \text{by \remref{rem1}}\\
& = (-1)^{2^m-1} (q^{k-1} u-1) \\
& = 1-q^{k-1} u.
\end{split}
\end{equation*}

Suppose $\binom{p-1}{2^m} \equiv 1 \quad \mod \, \, \, p$, then  $\binom{p-1}{2^m} \equiv 1 \quad \mod \, \, \, q^k$. Then, we would have
$1-q^{k-1} u \equiv 1 \quad \mod \, \, \, q^k$, i.e., $u \equiv 0 \quad \mod \, \, \, q$. However, this contradicts with $\gcd(q,u)=1$.
Thus, we proved that if $p$ is not a prime number, then  $\binom{p-1}{2^m}   \not \equiv 1 \quad \mod \, \, \, p \text{ for some $m \in \{ 1, \, 2, \, \dots, M \}$}.$

Hence, the proof of the theorem is completed.
\end{proof}

\begin{remark}\label{rem2}
When we work with the smallest integer $M$ with $\sqrt{p} \leq 2^M \leq p-1$, taking the logarithm gives $\frac{1}{2}\log_{2}p =\log_{2}\sqrt{p} \leq \log_{2}2^M=M$. Therefore, the smallest such integer $M$ is $\lceil \frac{1}{2}\log_{2}p \rceil$.
\end{remark}

Here are some examples:

We have $1024=2^{10} < 1093 < 2^{11}=2048$, $\binom{1093^2-1}{2^k} \equiv 1 \quad \mod \, \, \, 1093^2$ for each integer $1 \leq k \leq 10$, but $1093^2$ fails to be a prime number by \thmref{thm p-1 main} as $\binom{1093^{2}-1}{2^{11}} \equiv 1193557 \not \equiv 1$  modulo $1093^2$. Note that $1093^2=1194649$ is a Catalan pseudoprime.

We have $\binom{1194659-1}{2^k} \equiv 1 \quad \mod \, \, \, 1194659$ for each integer $1 \leq k \leq 11$ and $2^{11}>\sqrt{1194659} \approx 1093.00457>2^{10}$, so $1194659$ is a prime number by \thmref{thm p-1 main}.

\begin{theorem}\label{thm p-1 main2} Let $p>4$ be an integer, and let $M$ be the smallest integer with $\sqrt{p} \leq 2^M \leq p-1$. We have
$$\displaystyle{\text{$p$ is a prime number if and only if } \binom{p-1}{2^m-1}   \equiv -1 \quad \mod \, \, \, p \text{ for each $m \in \{ 1, \, 2, \, \dots, M \}$}.}$$
\end{theorem}
\begin{proof}
When $p$ is a prime,  $\binom{p-1}{2^m-1}  \equiv -1 \quad \mod \, \, \, p$  by \thmref{thm Leibniz2}.


For an odd number $p$, suppose $\binom{p-1}{2^m-1}   \equiv -1 \quad \mod \, \, \, p \text{ for each $m \in \{ 1, \, 2, \, \dots, M \}$.}$ Then following the same 
arguments used in the proof of \thmref{thm p-1 main} gives the result in this case.

For an even number $p=2a$ with a positive integer $a$ such that $\gcd(a,3)=1$, we have $M \geq 2$. Then for $m=2$, 
$\binom{p-1}{2^m-1} = \binom{2a-1}{3} \equiv a-1 \quad \mod \, \, \, 2a$. Therefore, the congruence $\binom{p-1}{2^2-1} \equiv -1 \quad \mod \, \, \, p$ leads to $a \equiv 0 \quad \mod \, \, \, 2a$, which is not possible.

As the last case, if $p=6a$ for some positive integer $a$, then $M \geq 2$ and for $m=2$ we have 
$\binom{p-1}{2^m-1} = \binom{6a-1}{3} \equiv 5a-1 \quad \mod \, \, \, 6a$. Therefore, the congruence $\binom{p-1}{2^2-1} \equiv -1 \quad \mod \, \, \, p$ leads to $5a \equiv 0 \quad \mod \, \, \, 6a$, which is impossible.
\end{proof}

\begin{theorem}\label{thm p-1 main3} Let $p>5$ be an integer, and let $M$ be the smallest integer with $\sqrt{p} \leq 2^M \leq p-1$. We have
$$\displaystyle{\text{$p$ is a prime number if and only if } \binom{p-1}{2^m+1}   \equiv -1 \quad \mod \, \, \, p \text{ for each $m \in \{ 1, \, 2, \, \dots, M \}$}.}$$
\end{theorem}
\begin{proof}
Arguments as in the proof of \thmref{thm p-1 main} give the result.
\end{proof}


Next, we consider the sufficient sets of congruences involving the binomial coefficients of type $\binom{p-2}{i}$.
\begin{theorem}\label{thm p-2 maina} Let $p>2$ be an odd integer, and let $M$ be the smallest integer with $\sqrt{p} \leq 2^M \leq p-1$. Then,
$p$ is a prime number if and only if  $\binom{p-2}{2^m-1}   \equiv - 2^m \quad \mod \, \, \, p$ for each $m \in \{ 1, \, 2, \, \dots, M \}$.
\end{theorem}
\begin{proof}
Since $\binom{p-1}{2^m}=\frac{p-1}{2^m}\binom{p-2}{2^m-1}$, we have 
\begin{equation}\label{eqn binom1}
\begin{split}
2^m \binom{p-1}{2^m}=(p-1)\binom{p-2}{2^m-1}.
\end{split}
\end{equation}
If $p$ is a prime number, then
$2^m \equiv -\binom{p-2}{2^m-1} \quad \mod \, \, \, p$ for each $m \in \{ 1, \, 2, \, \dots, M \}$ by \eqref{eqn binom1} and \thmref{thm p-1 main}. This is what we wanted to have.

Suppose $\binom{p-2}{2^m-1}   \equiv - 2^m \quad \mod \, \, \, p$ for each $m \in \{ 1, \, 2, \, \dots, M \}$. Then 
$2^m \binom{p-1}{2^m} \equiv (p-1)(-2^m) \quad \mod \, \, \, p$ by \eqref{eqn binom1}. Moreover, if $p$ is odd, this congruence can be written as  $\binom{p-1}{2^m} \equiv 1 \quad \mod \, \, \, p$ for each $m \in \{ 1, \, 2, \, \dots, M \}$. Then, $p$ is a prime number by  \thmref{thm p-1 main}.

This completes the proof of the theorem.
\end{proof}

Here are some examples:

Since $\sqrt{137}<12<2^4$ and $\binom{137-2}{2^1-1} \equiv -2$, $\binom{137-2}{2^2-1} \equiv -2^2$, $\binom{137-2}{2^3-1} \equiv -2^3$, $\binom{137-2}{2^4-1} \equiv -2^4$  modulo $137$, we conclude that $137$ is a prime.

Since $\sqrt{341}<19<2^5$ and $\binom{341-2}{2^4-1} \not \equiv -2^4 \quad \mod \, \, \, 341$, we conclude that $341$ is not a prime.

Although $\binom{19^2-2}{2^1-1} \equiv -2$, $\binom{19^2-2}{2^2-1} \equiv -2^2$, $\binom{19^2-2}{2^3-1} \equiv -2^3$, $\binom{19^2-2}{2^4-1} \equiv -2^4$ modulo $19^2$, $19^2$ fails to be prime as $\binom{19^5-2}{2^5-1} \not \equiv -2^5$  modulo $19^2$.

\begin{theorem}\label{thm p-2 mainb} Let $p>2$ be an integer, and let $M$ be the smallest integer with $\sqrt{p} \leq 2^M \leq p-1$. Then,
$p$ is a prime number if and only if $\binom{p-2}{2^m}   \equiv 2^m+1 \quad \mod \, \, \, p$ for each $m \in \{ 1, \, 2, \, \dots, M \}$.
\end{theorem}
\begin{proof}
\textbf{Case I ($\Rightarrow$):}

Since $\binom{p-1}{2^m+1}=\frac{p-1}{2^m+1}\binom{p-2}{2^m}$, we have 
\begin{equation}\label{eqn binom2}
\begin{split}
(2^m+1) \binom{p-1}{2^m+1}=(p-1)\binom{p-2}{2^m}.
\end{split}
\end{equation}
If $p$ is a prime number, then
$(2^m+1) (-1) \equiv -\binom{p-2}{2^m-1} \quad \mod \, \, \, p$ for each $m \in \{ 1, \, 2, \, \dots, M \}$ by \eqref{eqn binom2} and \thmref{thm p-1 main3}. This is what we wanted to show.

\textbf{Case II ($\Leftarrow$):}
The proof of this part is similar to the proof given in \thmref{thm p-1 main}.

We'll prove the contraposition. Namely, we prove that if $p$ is not a prime number, then $\binom{p-2}{2^m}   \not \equiv 2^m+1 \quad \mod \, \, \, p \text{ for some $m \in \{ 1, \, 2, \, \dots, M \}$}.$

If $p$ is not a prime number, then it has a prime divisor $q$ such that $q \leq \sqrt{p}$ by \propref{prop basic}. Then, $p=q^k u$ for some integers $k \geq 1$ and
$u \geq 1$ with $\gcd(q,u)=1$. Since $q \leq \sqrt{p}$, there is a positive integer $m \leq M$ such that $2^{m-1} < q \leq 2^m <  2^m+1 < 2q$. In particular,  $\gcd(q,2^m+1)=1$.
Then,
\begin{equation*}\label{eqn cong3b}
\begin{split}
\binom{p-2}{2^m} &= \frac{p-2}{2} \cdots \frac{p-q+1}{q-1} \cdot \frac{p-q}{q} \cdot \frac{p-q-1}{q+1}\cdots
\frac{p-2^{m}}{2^{m}} \cdot \frac{p-2^m-1}{1}\\
&= \frac{q^k u-2}{2} \cdots \frac{q^k u-q+1}{q-1} \cdot \frac{q^k u-q}{q} \cdot \frac{q^k u-q-1}{q+1}\cdots
\frac{q^k u-2^{m}}{2^{m}} \cdot (q^k u-2^m-1)\\
& \equiv (-1) \dots (-1) . (q^{k-1} u-1). (-1) \dots (-1) . (-2^m-1) \quad \mod \, \, \, q^k \quad \text{by \remref{rem1}}\\
& = (-1)^{2^m-2} (q^{k-1} u-1) . (-2^m-1) \\
& = (1-q^{k-1} u) . (2^m+1).\\
\end{split}
\end{equation*}
Suppose $\binom{p-2}{2^m} \equiv 2^m+1 \quad \mod \, \, \, p$, then  $\binom{p-2}{2^m} \equiv 2^m+1 \quad \mod \, \, \, q^k$. Then, we would have
$(1-q^{k-1} u). (2^m+1) \equiv  2^m+1 \quad \mod \, \, \, q^k$. Since  $\gcd(q,2^m+1)=1$, $1-q^{k-1} u \equiv  1 \quad \mod \, \, \, q^k$,
i.e., $u \equiv 0 \quad \mod \, \, \, q$. However, this contradicts with $\gcd(q,u)=1$.
Thus, we proved that if $p$ is not a prime number, then  $\binom{p-2}{2^m}   \not \equiv 2^m+1 \quad \mod \, \, \, p \text{ for some $m \in \{ 1, \, 2, \, \dots, M \}$}.$
%

This completes the proof of the theorem.
\end{proof}

\begin{theorem}\label{thm p-2 mainc} Let $p>5$ be an odd integer, and let $M$ be the smallest integer with $\sqrt{p} \leq 2^M \leq p-1$. Then,

$p$ is a prime number if and only if $\binom{p-2}{2^m+1}   \equiv -2^m-2 \quad \mod \, \, \, p$ for each $m \in \{ 1, \, 2, \, \dots, M \}$.
\end{theorem}
\begin{proof}
If $p$ is a prime, for any $m \in \{ 1, \, 2, \, \dots, M \}$ we have
$\binom{p-1}{2^m+1}  \equiv -1 \quad \mod \, \, \, p$ by \thmref{thm p-1 main3} and $\binom{p-2}{2^m}  \equiv 2^m+1 \quad \mod \, \, \, p$ by \thmref{thm p-2 mainb}. Therefore, using the Pascal's identity
$\binom{p-1}{2^m+1}=\binom{p-2}{2^m}+\binom{p-2}{2^m+1}$ gives $ \binom{p-2}{2^m+1}   \equiv -2^m-2 \quad \mod \, \, \, p$.

If $p$ is an odd integer divisible by $3$, then $p=6a+3$ for some positive integer $a$, then $M \geq 2$ and for $m=1$ we have 
$\binom{p-2}{2^m+1} = \binom{6a+1}{3} \equiv 2a-3 \quad \mod \, \, \, 6a+3$. Therefore, the congruence $\binom{p-2}{2^1+1} \equiv -2^1-2 \quad \mod \, \, \, p$ leads to $2a+1 \equiv 0 \quad \mod \, \, \, 6a+3$, which is impossible.

If $p$ is a composite odd integer with $\gcd(p,3)=1$, then 
it has an odd prime divisor $q$ such that  $\gcd(q,3)=1$ and $q \leq \sqrt{p}$ by \propref{prop basic}. Then, $p=q^k u$ for some integers $k \geq 1$ and
$u \geq 1$ with $\gcd(q,u)=1$. Since $q \leq \sqrt{p}$, there is a positive integer $m \leq M$ such that $2^{m-1} < q \leq 2^m < 2q$. 
At this point, there are two possibilities: Either $2q=2^m+2$, i.e., $q=2^{m-1}+1$ or $2^m+2<2q$. 

\textbf{Case I:} $2^m+2<2q$. Then, $\gcd(q,2^m+2)=1$ and
\begin{equation*}\label{eqn cong3c}
\begin{split}
\binom{p-2}{2^m+1} &= \frac{p-2}{2} \cdots \frac{p-q}{q} \cdot \frac{p-q-1}{q+1}\cdots
\frac{p-2^{m}-1}{2^{m}+1} \cdot \frac{p-2^m-2}{1}\\
&= \frac{q^k u-2}{2} \cdots \frac{q^k u-q}{q} \cdot \frac{q^k u-q-1}{q+1}\cdots
\frac{q^k u-2^{m}-1}{2^{m}+1} \cdot (q^k u-2^m-2)\\
& \equiv (-1) \dots (-1) . (q^{k-1} u-1). (-1) \dots (-1) . (-2^m-2) \quad \mod \, \, \, q^k \quad \text{by \remref{rem1}}\\
& = (-1)^{2^m-1} (q^{k-1} u-1) . (-2^m-2) \\
& = (1-q^{k-1} u) . (-2^m-2).\\
\end{split}
\end{equation*}
Suppose $\binom{p-2}{2^m+1} \equiv -2^m-2 \quad \mod \, \, \, p$, then  $\binom{p-2}{2^m+1} \equiv -2^m-2 \quad \mod \, \, \, q^k$. Then, we would have
$(1-q^{k-1} u). (-2^m-2) \equiv  -2^m-2 \quad \mod \, \, \, q^k$. Since  $\gcd(q,2^{m}+2)=1$, $1-q^{k-1} u \equiv  1 \quad \mod \, \, \, q^k$,
i.e., $u \equiv 0 \quad \mod \, \, \, q$. However, this contradicts with $\gcd(q,u)=1$.

\textbf{Case II:} $q=2^{m-1}+1$. Then
\begin{equation*}\label{eqn cong3d}
\begin{split}
\binom{p-2}{2^{m-1}+1} &= \frac{p-2}{2} \cdots \frac{p-q+1}{q-1} \cdot \frac{p-q}{q} \cdot \frac{p-q-1}{1}\\
&= \frac{q^k u-2}{2} \cdots \frac{q^k u-q+1}{q-1} \cdot \frac{q^k u-q}{q} \cdot (q^k u-q-1)\\
& \equiv (-1) \dots (-1) . (q^{k-1} u-1). (-q-1) \quad \mod \, \, \, q^k \quad \text{by \remref{rem1}}\\
& = (-1)^{2^{m-1}-1} (q^{k-1} u-1) . (-2^{m-1}-2) \\
& = (1-q^{k-1} u) . (-2^{m-1}-2).\\
\end{split}
\end{equation*}
Suppose $\binom{p-2}{2^{m-1}+1} \equiv -2^{m-1}-2 \quad \mod \, \, \, p$, then  $\binom{p-2}{2^{m-1}+1} \equiv -2^{m-1}-2 \quad \mod \, \, \, q^k$. Then, we would have
$(1-q^{k-1} u). (-2^{m-1}-2) \equiv  -2^{m-1}-2 \quad \mod \, \, \, q^k$. Since  $q=2^{m-1}+1$ and $\gcd(q^k,q+1)=1$, we have $1-q^{k-1} u \equiv  1 \quad \mod \, \, \, q^k$,
i.e., $u \equiv 0 \quad \mod \, \, \, q$. However, this again contradicts with $\gcd(q,u)=1$.

That is, in both cases we proved that if $p$ is not a prime number, then  $\binom{p-2}{2^m+1}   \not \equiv -2^m-2 \quad \mod \, \, \, p \text{ for some $m \in \{ 1, \, 2, \, \dots, M \}$}.$

This completes the proof of the theorem.
\end{proof}

Since $2^3<\sqrt{241}<2^4$, $\binom{241-2}{2^1+1} \equiv -4$, $\binom{241-2}{2^2+1} \equiv -6$,  $\binom{241-2}{2^3+1} \equiv -10$ and  $\binom{241-2}{2^4+1} \equiv -18$  modulo $241$, we conclude that $241$ is a prime number.

Note that $2^4 < \sqrt{17^2}<2^5$ and $17=2^4+1$.
Although $\binom{17^2-2}{2^m+1} \equiv -2^m-2$ for $m \in \{ 1, \, 2, \, 3, \, M=5 \}$,   
$17^2$ fails to be a prime number as $\binom{17^2-2}{2^4+1} \equiv -1 \not \equiv -2^4-2$  modulo $17^2$.

Note that all the cases we consider obey the following general congruence for a prime number $p$, integers $ 1 \leq a \leq n$ and $0 \leq b \leq n-a$ \cite[Corollary 2.3]{C1}:
\begin{equation}\label{eqn cong general}
\begin{split}
\binom{p-a}{b} \equiv (-1)^b \binom{a+b-1}{b} \quad \mod \, \, \, p.
\end{split}
\end{equation}
However, the converse of this is does not hold in general. In this paper, we actually provide sufficient conditions for specific choices of $a$ and $b$ so that the converse becomes true.

Finally, we consider the sufficient sets of congruences involving the binomial coefficients of type $\binom{p-3}{i}$.
\begin{theorem}\label{thm p-3 main1} Let $p>4$ be an integer, and let $M$ be the smallest integer with $\sqrt{p} \leq 2^M \leq p-1$. Then,
$p$ is a prime number if and only if $\binom{p-3}{2^m-1}   \equiv -2^{m-1}(2^m+1) \quad \mod \, \, \, p$ for each $m \in \{ 1, \, 2, \, \dots, M \}$.
\end{theorem}
\begin{proof}
If $p$ is a prime number, then $\binom{p-3}{2^m-1} \equiv -\binom{2^m+1}{2} = -2^{m-1}(2^m+1) \quad \mod \, \, \, p$ 
for each $m \in \{ 1, \, 2, \, \dots, M \}$ by \eqref{eqn cong general}.

Since $\binom{p-2}{2^m}=\frac{p-2}{2^m}\binom{p-3}{2^m-1}$, we have 
\begin{equation}\label{eqn binom3}
\begin{split}
2^m \binom{p-2}{2^m}=(p-2)\binom{p-3}{2^m-1}.
\end{split}
\end{equation}

When $p>4$ and $\binom{p-3}{2^m-1}   \equiv -2^{m-1}(2^m+1) \quad \mod \, \, \, p$ for each $m \in \{ 1, \, 2, \, \dots, M \}$, we have
$2^m \binom{p-2}{2^m} \equiv 2^{m}(2^m+1)  \quad \mod \, \, \, p$ by \eqref{eqn binom3}. 
If $p$ is an odd integer, then this can be written as $\binom{p-2}{2^m} \equiv (2^m+1)  \quad \mod \, \, \, p$ for each $m \in \{ 1, \, 2, \, \dots, M \}$. Therefore, $p$ is a prime number by \thmref{thm p-2 mainb}.

If $p>4$ is an even integer, we can show that $\binom{p-3}{2^m-1}   \not \equiv -2^{m-1}(2^m+1) \quad \mod \, \, \, p$ for $m=2$, i.e., 
$\binom{p-3}{3}   \not \equiv -10 \quad \mod \, \, \, p$. Such $p$ is in one of the three forms: $6a$, $6a+2$ or $6a+4$ for some positive integer $a$.

If $p=6a$, $\binom{p-3}{3}  \equiv -10 \quad \mod \, \, \, p$ means that $5a \equiv 0 \quad \mod \, \, \, 6a$ which is not possible.

If $p=6a+2$, $\binom{p-3}{3}  \equiv -10 \quad \mod \, \, \, p$ means that $3a+1 \equiv 0 \quad \mod \, \, \, 6a+2$ which can not happen.

If $p=6a+4$, $\binom{p-3}{3}  \equiv -10 \quad \mod \, \, \, p$ means that $3a+2 \equiv 0 \quad \mod \, \, \, 6a+4$ which is contradictory.

This completes the proof.
\end{proof}

\begin{theorem}\label{thm p-3 main3} Let $p>5$ be an integer, and let $M$ be the smallest integer with $\sqrt{p} \leq 2^M \leq p-1$. Then,
$p$ is a prime number if and only if $\binom{p-3}{2^m+1}   \equiv -(2^{m-1}+1)(2^m+3) \quad \mod \, \, \, p$ for each $m \in \{ 1, \, 2, \, \dots, M \}$.
\end{theorem}
\begin{proof}
If $p$ is a prime number, then $\binom{p-3}{2^m+1} \equiv -\binom{2^m+3}{2^m+1} = -(2^{m-1}+1)(2^m+3) \quad \mod \, \, \, p$ 
for each $m \in \{ 1, \, 2, \, \dots, M \}$ by \eqref{eqn cong general}.

If $p$ is a composite integer, then 
it has a prime divisor $q \leq \sqrt{p}$ by \propref{prop basic}. Then, $p=q^k u$ for some integers $k \geq 1$ and
$u \geq 1$ with $\gcd(q,u)=1$. Since $q \leq \sqrt{p}$, there is a positive integer $m \leq M$ such that $2^{m-1} < q \leq 2^m < 2q$. 
We consider the following three cases: $q=2$, $q=2^{m-1}+1>2$ and $q>2^{m-1}+1$.

\textbf{Case I:} $q>2^{m-1}+1$. Then,
\begin{equation*}\label{eqn cong p-3 maind}
\begin{split}
\binom{p-3}{2^m+1} &= \frac{p-3}{3} \cdots \frac{p-q}{q} \cdots
\frac{p-2^{m}-1}{2^{m}+1} \cdot \frac{p-2^m-2}{2}  \cdot \frac{p-2^m-3}{1}\\
&= \frac{q^k u-3}{3} \cdots \frac{q^k u-q}{q}  \cdots
\frac{q^k u-2^{m}-1}{2^{m}+1} \cdot \frac{q^k u-2^m-2}{2} \cdot \frac{q^k u-2^m-3}{1}\\
& \equiv (-1) \dots (-1) . (q^{k-1} u-1). (-1) \dots  (-1) .(-2^{m-1}-1)  (-2^m-3) \quad \mod \, \, \, q^k \\
\end{split}
\end{equation*}
by \remref{rem1}. Thus,
\begin{equation*}\label{eqn cong p-3 maine}
\begin{split}
\binom{p-3}{2^m+1} & \equiv (-1)^{2^m-2} (q^{k-1} u-1) (-2^{m-1}-1)  (-2^m-3) \quad \mod \, \, \, q^k \\
& = (q^{k-1} u-1)  (2^{m-1}+1)(2^m+3).\\
\end{split}
\end{equation*}

Suppose $\binom{p-3}{2^{m}+1} \equiv -(2^{m-1}+1)(2^m+3) \quad \mod \, \, \, p$, then  $\binom{p-3}{2^{m}+1} \equiv -(2^{m-1}+1)(2^m+3) \quad \mod \, \, \, q^k$. Then, we would have
$(q^{k-1} u-1). (2^{m-1}+1) (2^m+3) \equiv  -(2^{m-1}+1)(2^m+3) \quad \mod \, \, \, q^k$. Since
both $2^{m-1}+1$ and $2^{m}+3$ are relatively prime to $q$, 
we have $q^{k-1} u-1 \equiv  -1 \quad \mod \, \, \, q^k$,
i.e., $u \equiv 0 \quad \mod \, \, \, q$. However, this is not the case as $\gcd(q,u)=1$.

\textbf{Case II:} $q=2^{m-1}+1$ and $m>1$. Then
\begin{equation*}\label{eqn cong p-3 mainf}
\begin{split}
&\binom{p-3}{2^{m-1}+1} = \frac{p-3}{3} \cdots \frac{p-q+1}{q-1} \cdot \frac{p-q}{q} \cdot \frac{p-q-1}{2} \cdot \frac{p-q-2}{1}\\
&= \frac{q^k u-3}{3} \cdots \frac{q^k u-q+1}{q-1} \cdot \frac{q^k u-q}{q} \cdot \frac{(q^k u-q-1)}{2} (q^k u-q-2) \\
& \equiv (-1) \dots (-1) . (q^{k-1} u-1). (-2^{m-2}-1)(-2^{m-1}-3) \quad \mod \, \, \, q^k \quad \text{by \remref{rem1}}\\
& = (-1)^{2^{m-1}-2} (q^{k-1} u-1) .  (-2^{m-2}-1)(-2^{m-1}-3)  \\
& = (q^{k-1} u-1) . (2^{m-2}+1)(2^{m-1}+3).\\
\end{split}
\end{equation*}
Suppose $\binom{p-3}{2^{m-1}+1} \equiv -(2^{m-2}+1)(2^{m-1}+3) \quad \mod \, \, \, p$, then  $\binom{p-3}{2^{m-1}+1} \equiv -(2^{m-2}+1)(2^{m-1}+3) \quad \mod \, \, \, q^k$. Then, we would have
$ (q^{k-1} u-1) . (2^{m-2}+1)(2^{m-1}+3) \equiv  -(2^{m-2}+1)(2^{m-1}+3) \quad \mod \, \, \, q^k$. Since  $q=2^{m-1}+1$ and 
$\gcd(q^k,(2^{m-2}+1)(2^{m-1}+3))=1$, we have $q^{k-1} u-1 \equiv  -1 \quad \mod \, \, \, q^k$,
i.e., $u \equiv 0 \quad \mod \, \, \, q$. However, this again contradicts with $\gcd(q,u)=1$.

\textbf{Case III:} $q=2$. Then $p$ is even, so $p=6a$, $p=6a+2$ or $p=6a+4$ for some positive integer $a$.
In each of these cases, considering the congruence $\binom{p-3}{2^m+1}   \equiv -(2^{m-1}+1)(2^m+3) \quad \mod \, \, \, p$ for $m=1$ leads to the following contradictions:

$\binom{6a-3}{3}   \equiv -10 \quad \mod \, \, \, 6a$, which reduces to $5a \equiv 0 \quad \mod \, \, \, 6a$.

$\binom{6a-1}{3}   \equiv -10 \quad \mod \, \, \, 6a+2$, which leads to $3a+1 \equiv 0 \quad \mod \, \, \, 6a+2$.

$\binom{6a+1}{3}   \equiv -10 \quad \mod \, \, \, 6a+4$, which results in $3a+2 \equiv 0 \quad \mod \, \, \, 6a+4$.

In each cases, we showed that if $\binom{p-3}{2^m+1}   \equiv -(2^{m-1}+1)(2^m+3) \quad \mod \, \, \, p$ for each $m \in \{ 1, \, 2, \, \dots, M \}$, then $p$ must be a prime number.

This completes  the proof.
\end{proof}

\begin{theorem}\label{thm p-3 main2} Let $p>5$ be an odd integer, and let $M$ be the smallest integer with $\sqrt{p} \leq 2^M \leq p-1$. Suppose that if $\gcd(p,2^m+1)=1$ for each integer $m \in \{ 1, \, 2, \, \dots, M \}$. Then,
$p$ is a prime number if and only if $\binom{p-3}{2^m}   \equiv (2^{m-1}+1)(2^m+1) \quad \mod \, \, \, p$ for each $m \in \{ 1, \, 2, \, \dots, M \}$.
\end{theorem}
\begin{proof}
If $p$ is a prime number, then $\binom{p-3}{2^m} \equiv \binom{2^m+2}{2^m} = (2^{m-1}+1)(2^m+1) \quad \mod \, \, \, p$ 
for each $m \in \{ 1, \, 2, \, \dots, M \}$ by \eqref{eqn cong general}.

Since $\binom{p-2}{2^m+1}=\frac{p-2}{2^m+1}\binom{p-3}{2^m}$, we have 
\begin{equation}\label{eqn binom4}
\begin{split}
(2^m+1) \binom{p-2}{2^m+1}=(p-2)\binom{p-3}{2^m}.
\end{split}
\end{equation}

When $p>5$ and $\binom{p-3}{2^m}   \equiv (2^{m-1}+1)(2^m+1) \quad \mod \, \, \, p$ for each $m \in \{ 1, \, 2, \, \dots, M \}$, we have
$(2^m+1) \binom{p-2}{2^m+1} \equiv -(2^{m}+2)(2^m+1)  \quad \mod \, \, \, p$ by \eqref{eqn binom4}. 
For each $m \in \{ 1, \, 2, \, \dots, M \}$,
if $\gcd(p,2^m+1)=1$, then we obtain $\binom{p-2}{2^m+1} \equiv -(2^m+2)  \quad \mod \, \, \, p$ . 
Moreover, if $p$ is an odd integer, $p$ is a prime number by \thmref{thm p-2 mainc}.

\end{proof}

There are more hypothesis in \thmref{thm p-3 main2} than the other cases we studied. Next, we give examples showing that these hypothesis are essential.

Note that $17=2^4+1$. Although $2^4<\sqrt{17^2}<2^5$ and $\binom{17^2-3}{2^k} \equiv (2^{k-1}+1)(2^k+1)$  modulo $17^2$ for each $k=1, \, 2, \, 3, \, 4, \, 5$, we can not conclude that $17^2$ is a prime as $\gcd(17^2,2^4+1) \neq 1$.

We observe similar behavior for known primes of the form $2^{2^k}+1$. Such primes are called Fermat's primes. Namely, we have
$\binom{p^2-3}{2^m} \equiv (2^{m-1}+1)(2^m+1)$  modulo $p^2$ for each $m \in \{ 1, \, 2, \, \dots, M \}$ and $p \in 
\{ 3, 5, 17, 257, 65537  \}$. The largest known Fermat's prime is $65537$. Based on our computations we conjecture that the same holds for all Fermat's primes. 
In cases like these, we need to increase the number of congruence conditions to ensure the primality of $p$. For example,

$\binom{5^2-3}{2^k} \equiv (2^{k-1}+1)(2^k+1)$  modulo $5^2$ fails for $k=4$,

$\binom{17^2-3}{2^k} \equiv (2^{k-1}+1)(2^k+1)$  modulo $17^2$ fails for $k=6$, 

$\binom{257^2-3}{2^k} \equiv (2^{k-1}+1)(2^k+1)$  modulo $257^2$ fails for $k=10$, and

$\binom{65537^2-3}{2^k} \equiv (2^{k-1}+1)(2^k+1)$  modulo $65537^2$ fails for $k=18$.

As the value of the integer $a$ gets bigger in \eqref{eqn cong general}, the number of $b$ giving binomial congruences 
to guarantee the primality of $p$ is expected to grow. In the examples above, we observed that when $a=3$, we need $\lceil \frac{1}{2}\log_{2} n \rceil +1$ number of $b$ values of the form $2^m$.

\end{document}